\documentclass{amsart}
\usepackage{amssymb}
\usepackage{amsfonts}

\newtheorem{theorem}{Theorem}
\theoremstyle{plain}

\newtheorem{corollary}{Corollary}

\newtheorem{definition}{Definition}

\newtheorem{lemma}{Lemma}

\newtheorem{remark}{Remark}

\numberwithin{equation}{section}
\input{tcilatex}

\begin{document}
\title[On new fixed point theorems for multivalued mappings ]{On different
type of fixed point theorem for multivalued mappings via measure of
noncompactness }
\author{Nour el Houda Bouzara$^*$ }
\address{Mathematical Faculty, University of Science and Technology Houari
Boumediene, Bab-Ezzouar, 16111, Algiers, Algeria.\\
}
\email{bzr.nour@gmail.com\\
}
\author{Vatan Karakaya}
\curraddr{Department of Mathematical Engineering, Yildiz Technical
University, Davutpasa Campus, Esenler, 34220 Istanbul,Turkey}
\email{vkkaya@yahoo.com}

\begin{abstract}
In this paper by using the measure of noncompactness concept, we present new
fixed point theorems for multivalued maps. In further we introduce a new
class of mappings which are general than Meir-Keeler mappings. Finally, we
use these results to investigate the existence of weak solutions to an
Evolution differential inclusion with lack of compactness.
\end{abstract}

\maketitle

\section{Introduction}

Recently many papers have appeared about generalizations of Darbo's fixed
point and its applications. For example, in 2015 Aghajani and mursaleen \cite%
{meirkeeler} introduced the definition of a Meir Keeler condensing operator
and proved a theorem that guarantees the existence of fixed points for
single valued mappings and proved a fixed point theorem which extended the
well-known Darbo's and Meir Keeler fixed point theorems. Another
generalization is due to where they proved the existence of fixed points
under a more general condition than the contraction condition. It is
interesting to see what happened in the multivalued case and whether the
conditions still hold.

Our aim in this paper is to recall some essential concepts and results.
Then, we give a version of a Meir Keeler theorem for condensing multivalued
mappings, we also present some related results and applications. In the
third section, we present a version of theorems presented in \cite%
{meirkeeler} for multivalued mappings and some related results, we also
provide an application for this case.

Finally, in order to indicate their applicability, we choose one among the
previous theorems and we use it to study the existence of mild solutions for
a nonlocal differential evolution inclusion.

\section{Preliminaries}

In this section, we survey some definitions and preliminary facts for
measure of noncompactness and multivalued analysis which will be used in
this paper.

Let $\left( X,d\right) $ and $\left( Y,d^{\prime }\right) $ be two metric
spaces. We use the following notations, 
\begin{equation*}
\mathcal{P}_{cl}(X)=\{A\in \mathcal{P}(X):A\ closed\},\quad \mathcal{P}%
_{b}(X)=\{A\in \mathcal{P}(X):A\ bounded\},
\end{equation*}%
\begin{equation*}
\mathcal{P}_{cv}(A)=\{A\in \mathcal{P}(X):A\ convex\},\quad \mathcal{P}%
_{cp}(X)=\{A\in \mathcal{P}(X):Y\ compact\}.
\end{equation*}%
Consider $H_{d}:\mathcal{P}(X)\times \mathcal{P}(X)\longrightarrow \mathbb{R}%
_{+}\cup \{\infty \}$, given by 
\begin{equation*}
H_{d}(\mathcal{A},\mathcal{B})=\displaystyle\max \left\{ \ \displaystyle%
\sup_{a\in \mathcal{A}}\ d(a,\mathcal{B})\ ,\ \displaystyle\sup_{b\in 
\mathcal{B}}\ d(\mathcal{A},b)\ \right\} ,
\end{equation*}%
where $d(\mathcal{A},b)=\displaystyle\inf_{a\in \mathcal{A}}\ d(a,b)$, $d(a,%
\mathcal{B})=\displaystyle\inf_{b\in \mathcal{B}}\ d(a,b)$. Then $(\mathcal{P%
}_{b,cl}(X),H_{d})$ is a metric space and $(\mathcal{P}_{cl}(X),H_{d})$ is a
generalized (complete) metric space (see \cite{Ki}).

We denote by%
\begin{equation*}
S_{F}\left( y\right) =\{f\in L^{1}(J,X):f(t)\in F(t,y_{\rho (t,y_{t})})\ ,\
a.e.\ t\in J\},
\end{equation*}%
the set of \textit{selectors} of $F$.

$C(E;X)$ is the Banach space of all continuous mappings from $E$ into $X$
with the norm 
\begin{equation*}
\Vert y\Vert =\sup \ \{\ |y(t)|:t\in E\ \}.
\end{equation*}%
$B(X)$ is the space of all bounded linear mappings $T$ from $X$ into $X$,
with the norm 
\begin{equation*}
\Vert F\Vert _{B(X)}=\sup \ \{\ |F(y)|\ :\ |y|=1\ \}.
\end{equation*}%
A multivalued map $F:X\rightarrow \mathcal{P}(X)$ has a fixed point if there
exists $x\in X$ such that $x\in F(x)$.

A multivalued map $F:X\rightarrow \mathcal{P}(X)$ is said to be convex
(closed)\ valued if $F\left( x\right) $ is convex (closed) in $Y$ for each
set $A\ $of $X$ and $F$ is bounded valued if $F\left( x\right) $ is bounded
in $Y$ for each $x\in X$, i,e 
\begin{equation*}
\sup_{x\in A}\left\{ \sup \{\ \Vert y\Vert \ :\ y\in F(x)\}\right\} <\infty .
\end{equation*}%
In further, $F$ is compact if $F(A)$ is relatively compact for every $B\in 
\mathcal{P}_{b}(X)$. Finally, F is upper semi-continuous (u.s.c.) on $X$ if
for each $x_{0}\in X$ the set $F(x_{0})$ is a nonempty, closed subset of $X$%
, and if for each open subset $U$ of $X$ containing $F(x_{0})$, there exists
an open neighborhood $V$ of $x_{0}$ such that $F\left( V\right) \subseteq U$.

\begin{lemma}
Assume that $D\subset X$ and $Fx$ is closed for all $x\in D$, then the
following conclusions hold,

\begin{enumerate}
\item[i)] if $F$ is u.s.c. and $D$ is closed, then $F$ has a closed graph
(i.e., $x_{n}\rightarrow x$ and $y_{n}\rightarrow y$ s.t $y_{n}\in F(x_{n})$ 
$\Rightarrow $ $y\in F(x)$).

\item[ii)] if $F\left( D\right) $ is compact and $D$ is closed, then $F$ is
u.s.c. if and only if $F$ has a closed graph.
\end{enumerate}
\end{lemma}

For more details on multivalued maps we refer to the books of Deimling \cite%
{De}, G\'{o}rniewicz \cite{Go}.

\begin{definition}
\cite{BaGo} Let $X$ be a Banach space and $\mathcal{B}_{X}$ the family of
bounded subset of $X.$ A map%
\begin{equation*}
\mu :\mathcal{B}_{X}\rightarrow \left[ 0,\infty \right)
\end{equation*}%
is called measure of noncompactness $\left( \text{MNC}\right) $ defined on $%
X $ if it satisfies the following

\begin{enumerate}
\item $\mu \left( A\right) =0\Leftrightarrow A$ is a precompact set.

\item $A\subset B\Rightarrow \mu \left( A\right) \leqslant \mu \left(
B\right) .$

\item $\mu \left( A\right) =\mu \left( \overline{A}\right) ,$ $\forall A\in 
\mathcal{B}_{X}.$

\item $\mu \left( ConvA\right) =\mu \left( A\right) .$

\item $\mu \left( \lambda A+\left( 1-\lambda \right) B\right) \leqslant
\lambda \mu \left( A\right) +\left( 1-\lambda \right) \mu \left( B\right) ,$
for $\lambda \in \left[ 0,1\right] .$

\item Let $\left( A_{n}\right) $ be a sequence of closed sets from $\mathcal{%
B}_{X}$\ such that $A_{n+1}\subseteq A_{n},$ $\left( n\geqslant 1\right) $
and $\lim\limits_{n\rightarrow \infty }\mu \left( A_{n}\right) =0.$ Then the
intersection set $A_{\infty }=\dbigcap\limits_{n=1}^{\infty }A_{n}$ is
nonempty and $A_{\infty }$ is precompact.
\end{enumerate}
\end{definition}

Let $\mu _{0}$ be the sequential measure of noncompactness generated by $\mu 
$, that is, for any bounded subset $A\subset X$, then%
\begin{equation*}
\mu _{0}\left( A\right) =\sup \left\{ \mu \left( x_{n}\right) _{n=1}^{\infty
}\text{ is a sequence in }A\right\} \text{.}
\end{equation*}%
The relation between $\mu _{0}$ and $\mu $ is given by the following
inequalities%
\begin{equation*}
\mu _{0}\left( A\right) \leqslant \mu \left( A\right) \leqslant 2\mu
_{0}\left( A\right) .
\end{equation*}%
However, if $X$ is a separable space then $\mu _{0}\left( A\right) =\mu
\left( A\right) .$

\begin{lemma}[\protect\cite{KaObZec}]
\label{equimnclem}%
\begin{equation*}
\end{equation*}

\begin{enumerate}
\item Let $A\subseteq C\left( E;X\right) $ is bounded, then $\mu \left(
A\left( t\right) \right) \leqslant \mu \left( A\right) $ for all $t\in E$,
where $A\left( t\right) =\left\{ y\left( t\right) ,y\in A\right\} \subset X$%
. Furthermore, if $A$ is equicontinuous on $E$, then $\mu \left( A\left(
t\right) \right) $ is continuous on $E$ and $\mu \left( A\right) =\sup
\left\{ \mu \left( A\left( t\right) \right) ,t\in E\right\} .$

\item If $A\subset C\left( E;X\right) $ is bounded and equicontinuous, then%
\begin{equation*}
\mu \left( \int_{0}^{t}A\left( s\right) ds\right) \leqslant \int_{0}^{t}\mu
\left( A\left( s\right) \right) ds,
\end{equation*}%
for all $t\in E,$ where $\int_{0}^{t}A\left( s\right) ds=\left\{
\int_{0}^{t}x\left( s\right) ds:x\in A\right\} $.
\end{enumerate}
\end{lemma}

\begin{lemma}[\protect\cite{Thi}]
\label{graph}Let $X$ be a Banach space and $F$ a Caratheodory multivalued
mapping. Let $\Phi :L^{1}\left( E;X\right) \rightarrow C\left( E;X\right) $
be linear continuous mapping. Then, 
\begin{eqnarray*}
\Phi \circ S_{F}:C\left( E;X\right) &\rightarrow &\mathcal{P}_{cl;c}\left(
C\left( E;X\right) \right) \\
u &\rightarrow &\left( \Phi \circ S_{F}\right) u:=\Phi \left( S_{F}\left(
u\right) \right) ,
\end{eqnarray*}%
is a closed graph operator in $C\left( E;X\right) \times C\left( E;X\right) $%
.
\end{lemma}

\begin{definition}[\protect\cite{akhmerov}]
A multivalued map $F:X\rightarrow \mathcal{P}(X)$ is called $k$-set
contraction multivalued mapping if there exists a constant $k,$ $0\leqslant
k<1$ such that%
\begin{equation*}
\mu (FA)\leqslant k\mu \left( A\right) \text{ for any bounded }A\subset X.
\end{equation*}%
If $\mu (FA)<k\mu \left( A\right) ,$ then $F$ is a condensing multivalued
mapping.
\end{definition}

\begin{theorem}
\label{condensing}Let $A$ be a closed convex and bounded subset of a Banach
space $X$ and let $F:A\rightarrow \mathcal{P}_{cl,cv}\left( A\right) $ be an
upper semi-continuous and condensing multivalued mapping. Then, $F$ has a
fixed point point.
\end{theorem}

The following results is due to Dhage \cite{dhage}.

\begin{theorem}
Let $A$ be a closed convex and bounded subset of a Banach space $X$ and let $%
F:A\rightarrow \mathcal{P}_{cl,cv}\left( A\right) $ be an upper
semi-continuous multivalued mapping such that%
\begin{equation*}
\mu (FW)\leqslant \varphi \left( \mu \left( W\right) \right) \text{ for any
bounded }W\subset A,
\end{equation*}%
where $\varphi :\mathbb{R}_{+}\rightarrow \mathbb{R}_{+}$ is a continuous
nondecreasing function that satisfies $\varphi \left( t\right) <t$.

Then, $F$ has a fixed point point and the set of fixed points is compact .
\end{theorem}

\begin{lemma}
Let $\psi :\mathbb{R}^{+}\rightarrow \mathbb{R}^{+}$ be a nondecreasing and
upper semi-continuous function. Then,

\begin{equation*}
\lim\limits_{n\rightarrow \infty }\psi ^{n}\left( t\right) =0\text{ }for%
\text{ }each\text{ }t>0\Leftrightarrow \ \psi \left( t\right) <t\text{ }for%
\text{ }any\text{ }t>0.
\end{equation*}
\end{lemma}

In what follows,we confine ourselves only to the fixed point theory related
to upper semicontinuous multi-valued mappings in Banach spaces. The first
fixed point theorem in this direction is due to Kakutani-Fan \cite{Fan}
which is as follows.

\begin{theorem}
Let $A$ be a nonempty compact convex subset of a Hausdorff locally convex
topological vector space E, and let $F:A\rightarrow \mathcal{P}%
_{cl,cv}\left( A\right) $ be an upper semicontinuous map. Then, $F$ has a
fixed point.
\end{theorem}

\section{Fixed point theorems for multivalued Meir-Keeler set contraction
mappings}

\begin{definition}
A Meir-Keeler condensing multivalued mapping if for each $\delta >0$ there
exists $\varepsilon >0$ such that%
\begin{equation*}
\varepsilon \leqslant \mu \left( A\right) <\varepsilon +\delta \Rightarrow
\mu (FA)<\varepsilon \text{.}
\end{equation*}
\end{definition}

\begin{remark}
The condensing multivalued mappings of Meir-Keeler type are more general
than condensing mappings. Indeed, let $F$ be a condensing mapping, that is, 
\begin{equation*}
\mu (FA)\leqslant k\mu \left( A\right) \text{ for any bounded }A\subset X%
\text{.}
\end{equation*}%
Suppose for $\delta =\left( \frac{1}{k}-1\right) \varepsilon $ that we have $%
\ \varepsilon \leqslant \mu \left( A\right) <\varepsilon +\delta $, then 
\begin{equation*}
\mu (FA)\leqslant k\mu \left( A\right) <\varepsilon +k\left( \frac{1}{k}%
-1\right) \varepsilon =\varepsilon .
\end{equation*}%
Thus, $F$ is a Meir Keeler condensing multivalued mapping.
\end{remark}

\begin{theorem}
\label{meir multi}Let $X$ be a Banach space and $A$ be a nonempty closed,
bounded and convex subset of $X$. Let $F:A\rightarrow \mathcal{P}%
_{cl,cv}\left( A\right) $ be multivalued upper semicontinuous mapping such
that for any bounded $W\subset A$, we have%
\begin{equation*}
\varepsilon \leqslant \mu \left( W\right) <\varepsilon +\delta \Rightarrow
\mu (FW)<\varepsilon \text{.}
\end{equation*}%
Then, $F$ has at least one fixed point in $A.$
\end{theorem}

\textbf{Proof. }Obviously, if we have%
\begin{equation*}
\varepsilon \leqslant \mu \left( W\right) <\varepsilon +\delta \Rightarrow
\mu (FW)<\varepsilon \text{,}
\end{equation*}%
then%
\begin{equation*}
\mu (FA)<\mu (A)\text{.}
\end{equation*}%
Thus by Theorem \ref{condensing}, $F$ has at least one fixed point.

\begin{corollary}
Let $X$ be a Banach space and $F:X\rightarrow \mathcal{P}\left( X\right) $
be multivalued mapping with convex values, closed graph and bounded range
such that, for any bounded $A\subset X$, we have%
\begin{equation*}
\mu (FA)\leqslant k\mu \left( A\right) ,\text{ for }0\leqslant k<1\text{.}
\end{equation*}%
Then, $F$ has at least one fixed point in $A.$
\end{corollary}

\section{Fixed point theorems for multivalued set contraction mappings of
Caristi type}

\begin{theorem}
Let $X$ be a Banach space and $A$ be a nonempty closed, bounded and convex
subset of $X$. Let $F:A\rightarrow \mathcal{P}_{cl,cv}\left( A\right) $ be
multivalued upper semi-continuous mapping such that for any bounded $%
W\subset A$, we have%
\begin{equation}
\psi \left( \mu \left( FW\right) \right) \leqslant \psi \left( \mu \left(
W\right) \right) -\varphi \left( \mu \left( W\right) \right) ,  \label{cd3.1}
\end{equation}%
where $\mu $ is an arbitrary measure of noncompactness and $\psi ,\varphi :%
\mathbb{R}_{+}\rightarrow \mathbb{R}_{+}$ are given functions such that $%
\varphi $ is lower semi-continuous and $\psi $ is continuous on $\mathbb{R}%
_{+}$. Moreover, $\varphi \left( 0\right) =0$ and $\varphi \left( t\right)
>0 $ for $t>0$. Then T has at least one fixed point in $A$.
\end{theorem}

\begin{proof}
Define the sequence $W_{0}=W$ and $W_{n+1}=\overline{co}\left( FW_{n}\right) 
$, clearly $\left( W_{n}\right) _{n\in \mathbb{N}}$ is a nonempty closed,
bounded, convex sequence and 
\begin{equation*}
W_{0}\subset W_{1}\subset ...\subset W_{n}\text{.}
\end{equation*}%
Since the sequence $\left( \mu \left( W_{n}\right) \right) _{n\in \mathbb{N}%
} $ is decreasing and bounded below $\left( \text{since }\mu \left(
W_{n}\right) >0\text{, }\forall n\in \mathbb{N}\right) $, then $\left( \mu
\left( W_{n}\right) \right) _{n\in \mathbb{N}}$ is a convergent sequence.
Put $\lim\limits_{n\rightarrow \infty }\mu \left( W_{n}\right) =l.$

In further, using properties of the measure of noncompactness we have,%
\begin{equation*}
\mu \left( W_{n+1}\right) =\mu \left( \overline{co}\left( FW_{n}\right)
\right) =\mu \left( FW_{n}\right) .
\end{equation*}

Then, in view of condition $\left( \text{\ref{cd3.1}}\right) $ we have%
\begin{eqnarray*}
\psi \left( \mu \left( W_{n+1}\right) \right) &=&\psi \left( \mu \left(
FW_{n}\right) \right) \\
&\leqslant &\psi \left( \mu \left( W_{n}\right) \right) -\varphi \left( \mu
\left( W_{n}\right) \right) .
\end{eqnarray*}%
By taking the limit sup we get%
\begin{equation*}
\lim_{n\rightarrow \infty }\sup \psi \left( \mu \left( W_{n+1}\right)
\right) \leqslant \lim_{n\rightarrow \infty }\sup \psi \left( \mu \left(
W_{n}\right) \right) -\lim_{n\rightarrow \infty }\inf \varphi \left( \mu
\left( W_{n}\right) \right) .
\end{equation*}%
Since $\psi $ is continuous and $\varphi $ is lower semi-continuous, we get%
\begin{equation*}
\psi \left( l\right) \leqslant \psi \left( l\right) -\varphi \left( l\right)
.
\end{equation*}%
Fellows that $\varphi \left( l\right) $ must be null, which means that $l=0.$
Thus 
\begin{equation*}
0=\lim_{n\rightarrow \infty }\sup \mu \left( W_{n}\right)
=\lim_{n\rightarrow \infty }\inf \mu \left( W_{n}\right) =\lim_{n\rightarrow
\infty }\mu \left( W_{n}\right) .
\end{equation*}%
Hence, using property 6. of measure of noncompactness we get $W_{\infty
}=\bigcap\limits_{n}W_{n}$ is compact. Then $F$ has at least one fixed point.
\end{proof}

\section{Existence of fixed points for multivalued power set contraction
mappings}

\begin{theorem}
Let $A$ be a nonempty closed, bounded and convex subset of a Banach space $X 
$ and $F:A\rightarrow \mathcal{P}_{cl,cv}\left( A\right) $ be a $k$-set
contraction mapping on $A$. Then, $N^{n}$ (for an integer $n>0$) is a $%
k^{n}- $set contraction on $A$.
\end{theorem}

\begin{proof}
Let $A$ be a nonempty closed, bounded and convex subset of $X$, then for any
bounded bounded $W\subset A$, 
\begin{eqnarray*}
\mu \left( N^{n}W\right) &=&\mu \left( N\left( N^{n-1}W\right) \right) \\
&\leqslant &k\mu \left( N^{n-1}W\right) \\
&\leqslant &k^{2}\mu \left( N^{n-2}W\right) \\
&&\vdots \\
&\leqslant &k^{n}\mu \left( W\right) \text{.}
\end{eqnarray*}%
Since $0\leqslant k<1$, hence $0\leqslant k^{n}<1$ and so $N^{n}$ is also a $%
k-$set contraction mapping.
\end{proof}

\begin{remark}
The inverse is not true that is if $N^{n}$ is a $k$-set contraction mapping
then $N$ could be not a k-set contraction mapping.
\end{remark}

\begin{theorem}
\label{generali}Let $A$ be a nonempty closed, bounded and convex subspace of
a Banach space $X$ and $F:A\rightarrow \mathcal{P}_{cl,cv}\left( A\right) $
be an upper semi-continuous multivalued mapping such that for any $%
n\geqslant 1$ we have $N^{n}\left( conv\left( W\right) \right) \subseteq
conv\left( N^{n}W\right) $ and%
\begin{equation}
\mu \left( N^{n}W\right) \leqslant k_{n}\mu \left( W\right) \text{, for any
bounded }W\subset A.  \label{ineq2}
\end{equation}
where $k_{n}\rightarrow 0,$ $n\rightarrow +\infty $. Then, there exists at
least one $x$ such that $x\in Fx$.
\end{theorem}

\begin{proof}
\label{prf1}Let the iteration $W_{0}=W$ and $\ W_{n}=\overline{conv}\left(
NA_{n-1}\right) .$ Obviously $\left( A_{n}\right) $ is a sequence of
nonempty closed, bounded and convex subsets of $A.$

It is clear that $\left( A_{n}\right) _{n}$ is decreasing.

Then, by using the properties of the measure of noncompactness, we get%
\begin{eqnarray*}
\mu \left( W_{n}\right) &=&\mu \left( conv\left( NW_{n-1}\right) \right) \\
&\leqslant &\mu \left( conv\left( NW_{n-1}\right) \right) =\mu \left(
NW_{n-1}\right) \\
&\leqslant &\mu \left( N\left( conv\left( NW_{n-2}\right) \right) \right) \\
&\leqslant &\mu \left( N\left( conv\left( NW_{n-2}\right) \right) \right) \\
&\leqslant &\mu \left( N^{2}\left( W_{n-2}\right) \right)
\end{eqnarray*}%
Repeating this process many times we get%
\begin{equation*}
\mu \left( W_{n}\right) \leqslant \mu \left( N^{n}\left( W_{0}\right)
\right) .
\end{equation*}%
Using Inequality \ref{ineq2}. we get $\mu \left( W_{n}\right) \leqslant \mu
\left( N^{n}\left( W_{0}\right) \right) \leqslant k_{n}\mu \left(
W_{0}\right) .$

By taking the limit, we get $\lim\limits_{n\rightarrow \infty }\mu \left(
W_{n}\right) =0,$ which implies that $W_{\infty }$ is compact. Hence $N$ has
at least one fixed point in $W_{\infty }\subset A.$
\end{proof}

\section{Application to Evolution differential inclusions with nonlocal
condition}

The multi-valued fixed point theorems of this paper can have some nice
applications to differential and integral inclusions as an example we choose
to provide an application for Theorem \ref{meir multi}. One can notice that
other applications can be given by changing the contractive condition which
the mappings is supposed to satisfy.

Let following evolution differential inclusions with nonlocal conditions%
\begin{eqnarray}
y^{\prime }(t) &\in &A(t)y(t)+F(t,y\left( t\right) ),\ \ t\in J:=[0,+\infty )
\label{ev1} \\
y(0) &=&\varphi (y),  \label{evcnd}
\end{eqnarray}

where $F$ is an upper Caratheodory multimap, $\varphi :C\left( J,X\right)
\rightarrow X$ is a given $X$-valued function. $\left\{ A\left( t\right)
:t\in J\right\} $ is a family of linear closed unbounded operators on $X$
with domain $D(A(t))$ independent of $t$ that generate an evolution system
of operators $\left\{ U\left( t,s\right) :t,s\in \Delta \right\} $ with $%
\Delta =\left\{ \left( t,s\right) \in J\times J:0\leqslant s\leqslant
t<\infty \right\} $.

The main work for this section is to study the existence of mild solutions
for this non-local inclusion.

Before we start studying this problem we recall some concepts and results
that will be needed through the section.

Define the set%
\begin{equation*}
S_{F}\left( y\right) =\left\{ f\in L^{1}\left( J,X\right) :f\left( t\right)
\in F\left( t,y\left( t\right) \right) \right\} .
\end{equation*}

\begin{definition}
A mapping $F:J\times C(J,X)\longrightarrow \mathcal{P}_{cp,cv}(X)$ is said
to be an upper Carath\'{e}odory multivalued map if it satisfies,

\begin{itemize}
\item[(i)] $x\mapsto F(t,x)$ is upper semi-continuous (with respect to the
metric $H_{d}$) for almost all $t\in J$.

\item[(ii)] $t\mapsto F(t,x)$ is measurable for each $x\in C(J,X)$.
\end{itemize}
\end{definition}

\begin{definition}
A family $\{U(t,s)\}_{(t,s)\in \Delta }$ of bounded linear operators $%
U(t,s):X\rightarrow X$ where $(t,s)\in \Delta :=\{(t,s)\in J\times
J:0\leqslant s\leqslant t<+\infty \}$ is called en evolution system if the
following properties are satisfied,

\begin{enumerate}
\item $U(t,t)=I$ where $I$ is the identity operator in $X$ and $U(t,s)\
U(s,\tau )=U(t,\tau )$ for $0\leqslant \tau \leqslant s\leqslant t<+\infty $,

\item The mapping $(t,s)\rightarrow U(t,s)\ y$ is strongly continuous, that
is, there exists a constant $C>0$ such that 
\begin{equation*}
\Vert U(t,s)\Vert \leqslant M\ \ \text{for any }(t,s)\in \Delta .
\end{equation*}
\end{enumerate}
\end{definition}

An evolution system $U(t,s)$ is said to be compact if $U(t,s)$ is compact
for any $t-s>0$. $U(t,s)$ is said to be equicontinuous if $\left\{
U(t,s)x:x\in M\right\} $ is equicontinuous at $0\leqslant s<t\leqslant b$
for any bounded subset $M\subset X$. Clearly, if $U(t,s)$ is a compact
evolution system, it must be equicontinuous. The inverse is not necessarily
true.

More details on evolution systems and their properties could be found on the
books of Ahmed \cite{Ah}, Engel and Nagel \cite{EnNa} and Pazy \cite{Pa}.

\begin{definition}
We say that the function $y(t)\in C\left( J,X\right) $ is a mild solution of
the evolution system $($\ref{ev1}$)-($\ref{evcnd}$)$ if it satisfies the
following integral equation 
\begin{equation}
y(t)=U(t,0)\ \varphi (y)+\int_{0}^{t}U(t,s)\ f(s)\ ds,  \label{evmild}
\end{equation}%
for all $t\in \mathbb{R}_{+}$ and $f\in S_{F}\left( y\right) .$
\end{definition}

Assume the following hypothesis which are needed thereafter :

\begin{itemize}
\item[$\left( H1\right) $] $\left\{ A\left( t\right) :t\in J\right\} $ is a
family of linear operators. $A\left( t\right) :D\left( A\right) \subset
X\rightarrow X$ generates an equicontinuous evolution system $\left\{
U\left( t,s\right) :\left( t,s\right) \in \Delta \right\} $ and%
\begin{equation*}
\left\vert U\left( t,s\right) \right\vert \leqslant M.
\end{equation*}
\end{itemize}

\begin{itemize}
\item[$(H2)$] The multifunction $F:J\times C(J;X)\longrightarrow \mathcal{P}%
_{cl,cv}(X)$ is an upper Carath\'{e}odory and $\varphi :C(J;X)\rightarrow X$
is continuous, if we have for any $\varepsilon >0$ there exists $\delta >0$
such that%
\begin{equation*}
\varepsilon \leqslant \mu \left( W\right) <\varepsilon +\delta \text{ for
any bounded }W\subset A
\end{equation*}%
implies 
\begin{equation*}
\mu \left( \varphi \left( W\right) \right) <\frac{\varepsilon }{2M}\text{
and }\mu \left( F\left( t,W\right) \right) <\frac{\varepsilon }{2Mt}\text{
for any }t\in J.
\end{equation*}

\item[$\left( H3\right) $] There exists a constant $r>0$ such that%
\begin{equation*}
M\left[ \left\Vert \varphi \left( y\right) \right\Vert +\left\{ \left\Vert
f\left( t\right) \right\Vert _{1}:f\in S_{F}\left( y\right) ,y\in
A_{0}\right\} \right] \leqslant r
\end{equation*}%
where, $A_{0}=\left\{ y\in C(J;X):\left\Vert y\left( t\right) \right\Vert
\leqslant r\text{ for all }t\in J\right\} $.
\end{itemize}

\begin{theorem}
Under the assumptions $\left( H1\right) -\left( H3\right) $ the non local
problem $\left( \text{\ref{evmild}}\right) -\left( \text{\ref{evcnd}}\right) 
$ has at least one mild solution in the space $C\left( J,X\right) $.
\end{theorem}

\textbf{Proof.} To solve problem $($\ref{e1}$)-($\ref{e2}$)$ we transform it
to the following fixed-point problem.

Consider the multivalued operator $N:C(J;X)\rightarrow \mathcal{P}(C(J;X))$
defined by, 
\begin{equation*}
N(y)=\left\{ h\in C(J;X):h(t)=U(t,0)\varphi (y)+\int_{0}^{t}U(t,s)\ f(s)\ ds,%
\text{ with }f\in S_{F}\left( y\right) \right\} .
\end{equation*}

We can notice that fixed points of the operator $N$ are mild solutions of
problem $($\ref{evmild}$)-($\ref{evcnd}$)$.

Clearly for each $y\in C([-r,+\infty );X)$, the set $S_{F}\left( y\right) $
is nonempty since, by $(H2)$, $F$ has a measurable selection (see \cite{CaVa}%
).

To prove that $N$ has a fixed point , we need to satisfy all the conditions
of one of above theorems, for example let choose Theorem \ref{generali}.

Let $A_{0}=\left\{ y\in C(J;X):\left\Vert y\left( t\right) \right\Vert
\leqslant r\text{ for all }t\in J\right\} $. Obviously, $A_{0}$ is closed,
bounded and convex.

To show that $NA_{0}\subseteq A_{0}$, we need first to prove that the family 
\begin{equation*}
\left\{ \int_{0}^{t}U(t,s)\ f(s)\ ds\ :f\in S_{F}\left( y\right) \text{ and }%
y\in A_{0}\right\} 
\end{equation*}
is equicontinuous for $t\in J$ that is all the functions are continuous and
they have equal variation over a given neighborhood.

In view of $\left( H1\right) $ we have that functions in the set $\left\{
U\left( t,s\right) :\left( t,s\right) \in \Delta \right\} $ are
equicontinuous, (i,e) for every $\varepsilon >0$ there exists $\delta >0$
such that $\left\vert t-\tau \right\vert <\delta $ implies $\left\Vert
U\left( t,s\right) -u\left( \tau ,s\right) \right\Vert <\varepsilon $ for
all $U\left( t,s\right) \in \left\{ U\left( t,s\right) :\left( t,s\right)
\in \Delta \right\} $

Then, given some $\varepsilon >0$ let $\delta =\frac{\varepsilon ^{\prime }}{%
\varepsilon \left\Vert f\right\Vert _{\infty }}$ such that $\left\vert
t-\tau \right\vert <\delta $, we have%
\begin{equation*}
\left\vert \int_{0}^{t}U(t,s)\ f(s)\ ds-\int_{0}^{\tau }U(\tau ,s)\ f(s)\
ds\right\vert \leqslant \int_{\tau }^{t}\left\vert U\left( t,s\right)
-U\left( \tau ,s\right) \right\vert \left\vert f\left( s\right) \right\vert
ds.
\end{equation*}%
Regarding the fact that $\left\{ U\left( t,s\right) :\left( t,s\right) \in
\Delta \right\} $ is equicontinuous then%
\begin{eqnarray*}
\left\vert \int_{0}^{t}U(t,s)\ f(s)\ ds-\int_{0}^{\tau }U(\tau ,s)\ f(s)\
ds\right\vert &\leqslant &\varepsilon \left\Vert f\right\Vert _{\infty
}\left\vert t-\tau \right\vert \\
&<&\varepsilon \left\Vert f\right\Vert _{\infty }\frac{\varepsilon ^{\prime }%
}{\varepsilon \left\Vert f\right\Vert _{\infty }}=\varepsilon ^{\prime }.
\end{eqnarray*}%
Hence we conclude that $\left\{ \int_{0}^{t}U(t,s)\ f(s)\ ds\ :f\in
S_{F}\left( y\right) \text{ and }y\in A_{0}\right\} $ is equicontinuous for $%
t\in J$.

Now, let show that $NA_{0}\subseteq A_{0}$. Let for $t\in J$,%
\begin{eqnarray*}
\left\vert h\left( t\right) \right\vert &=&\left\vert U\left( t,0\right)
\varphi (y)+\int_{0}^{t}U(t,s)\ f(s)\ ds\right\vert \\
&\leqslant &\left\vert U\left( t,0\right) \varphi (y)\right\vert
+\int_{0}^{t}\left\vert U(t,s)\ f(s)\right\vert ds \\
&\leqslant &M\left\Vert \varphi \left( y\right) \right\Vert +M\left\Vert
f\right\Vert _{1} \\
&=&M\left[ \left\Vert \varphi \left( y\right) \right\Vert +\left\Vert
f\right\Vert _{1}\right] \leqslant r,
\end{eqnarray*}%
thus $NA_{0}\subseteq A_{0}$.

In further it is easy to see that $N$ has convex valued.

Now let show that $N$ has a closed graph, let $y_{n}\rightarrow y$ and $%
h_{n}\rightarrow h$ such that $h_{n}\left( t\right) \in N\left( y_{n}\right) 
$ and let show that $h\left( t\right) \in N\left( y\right) .$

Then, there exists a sequence $f_{n}\in S_{F}\left( y_{n}\right) $ such that%
\begin{equation*}
h_{n}\left( t\right) =U\left( t,0\right) \varphi
(y_{n})+\int_{0}^{t}U(t,s)f_{n}(s)ds.
\end{equation*}%
Consider the linear operator $\Phi :L^{1}\left( J;X\right) \rightarrow
C\left( J;X\right) $ defined by 
\begin{equation*}
\Phi f\left( t\right) =\int_{0}^{t}U(t,s)f_{n}(s)ds.
\end{equation*}%
Clearly, $\Phi $ is linear and continuous. Then from Lemma \ref{graph}. we
get that $\Phi \circ S_{F}\left( y\right) $ is a closed graph operator. In
further, we have%
\begin{equation*}
h_{n}\left( \cdot \right) -U\left( t,0\right) \varphi (y_{n})\in \Phi \circ
S_{F,y}.
\end{equation*}%
Since $y_{n}\rightarrow y$ and $h_{n}\rightarrow h,$ then%
\begin{equation*}
h\left( \cdot \right) -U\left( t,0\right) \varphi (y_{n})\in \Phi \circ
S_{F,y}.
\end{equation*}%
That is, there exists a function $f\in S_{F}\left( y\right) $ such that%
\begin{equation*}
h\left( t\right) =U\left( t,0\right) \varphi (y)+\int_{0}^{t}U(t,s)f(s)ds.
\end{equation*}%
Therefore $N$ has a closed graph, hence $N$ has closed values on $C\left(
J;X\right) $.

Let $W$ be a bounded subset of $A$ such that 
\begin{equation*}
\varepsilon \leqslant \mu \left( W\right) <\varepsilon +\delta \text{.}
\end{equation*}%
We know that the family $\left\{ \int_{0}^{t}U(t,s)f(s)ds,f\in S_{F}\left(
W\left( t\right) \right) \right\} $ is equicontinuous, hence by Lemma \ref%
{equimnclem}, we have%
\begin{eqnarray*}
\mu \left( \int_{0}^{t}U(t,s)f(s)ds,\text{ }f\in S_{F}\left( W\left(
t\right) \right) \right) &\leqslant &\int_{0}^{t}\mu \left( U(t,s)f(s),\text{
}f\in S_{F}\left( W\left( t\right) \right) \right) ds \\
&\leqslant &M\int_{0}^{t}\mu \left( f(s),\text{ }f\in S_{F}\left( W\left(
t\right) \right) \right) ds \\
&\leqslant &Mt\mu \left( F\left( t,W\left( t\right) \right) \right) .
\end{eqnarray*}%
Therefore%
\begin{eqnarray*}
\mu \left( NW\right) &=&\mu N\left( U\left( t,0\right) \varphi (W\left(
t\right) )+\int_{0}^{t}U(t,s)f(s)ds,\text{ }f\in S_{F}\left( W\left(
t\right) \right) \right) \\
&\leqslant &\mu \left( U\left( t,0\right) \varphi (W\left( t\right) )\right)
+\mu \left( \int_{0}^{t}U(t,s)f(s)ds,\text{ }f\in S_{F}\left( W\left(
t\right) \right) \right) \\
&\leqslant &M\mu \left( \varphi (W\left( t\right) )\right) +Mt\mu \left(
F\left( t,W\left( t\right) \right) \right) .
\end{eqnarray*}%
In view of $\left( H2\right) $, we get%
\begin{equation*}
\mu \left( NW\left( t\right) \right) \leqslant M\frac{\varepsilon }{2M}+Mt%
\frac{\varepsilon }{2Mt}=\varepsilon .
\end{equation*}%
Therefore, for $\varepsilon \leqslant \mu <\varepsilon +\delta $ we obtained 
$\mu \left( NW\left( t\right) \right) \leqslant \varepsilon .$ Thus
regarding Theorem \ref{meir multi}, $N$ has at least one fixed point, hence
the problem $\left( \text{\ref{evmild}}\right) -\left( \text{\ref{evcnd}}%
\right) $.

\end{document}